\documentclass{amsart}
\usepackage{graphicx}
\usepackage{enumerate}
\usepackage{amssymb}
\usepackage{amsaddr}
\usepackage{url}
\newtheorem{thm}{Theorem}[section]
\newtheorem{cor}[thm]{Corollary}
\newtheorem{lem}[thm]{Lemma}
\newtheorem{prop}[thm]{Proposition}
\newtheorem{conj}[thm]{Conjecture}
\newtheorem{defn}[thm]{Definition}

\theoremstyle{definition}
\theoremstyle{remark}
\newtheorem{rem}[thm]{Remark}
\numberwithin{equation}{section}

\newcommand{\Pee}{\mathcal{P}}
\newcommand{\Text}[1]{\text{\textnormal{#1}}}
\newcommand{\per}{\text{\textnormal{per}}}
\newcommand{\Sym}{\text{\textnormal{Sym}}}

\newcommand{\id}{\text{\textnormal{id}}}
\begin{document}

\title{On extensions of the Alon-Tarsi Latin Square conjecture}
\author{Daniel Kotlar}
\address{Computer Science Department, Tel-Hai College, Upper Galilee 12210, Israel}%
\email{dannykot@telhai.ac.il}

\thanks{I Thank an anonymous reviewer for providing the proofs of Theorem~\ref{thm3:1} and Corollary~\ref{cor1}}%

\begin{abstract}
Expressions involving the product of the permanent with the $(n-1)^\Text{st}$ power of the determinant of a matrix of indeterminates, and of (0,1)-matrices, are shown to be related to an extension to odd dimensions of the Alon-Tarsi Latin Square Conjecture. These yield an alternative proof of a theorem of Drisko, stating that the extended conjecture holds for odd primes. An identity involving an alternating sum of permanents of (0,1)-matrices is obtained.
\end{abstract}
\maketitle
\section{Introduction}\label{section1}
 A \emph{Latin square} of order $n$ is an $n \times n$ array of numbers in $[n]:=\{1,\ldots,n\}$ so that each number appears exactly once in each row and each column. Let $L_n$ be the number of Latin squares of order $n$. Let $\Sym(n)$ be the symmetric group of permutations of $[n]$. For a permutation $\pi\in\Sym(n)$ we denote its sign by $\epsilon(\pi)$. Viewing the rows and columns of a Latin square $L$ as elements of $\Sym(n)$, the row-sign (column-sign) of $L$ is defined to be the product of the signs of the rows (columns) of $L$. The sign of $L$, denoted $\epsilon(L)$, is the product of the row-sign and the column-sign of $L$. The \emph{parity} of a Latin square is even (resp. odd) if its sign is 1 (resp. -1). The \emph{row parity} and \emph{column parity} of a Latin square are defined analogously. We denote by $L_n^{\Text{EVEN}}$ ($L_n^{\Text{ODD}}$) the number of even (odd) Latin squares of order $n$. The Alon-Tarsi Latin Square Conjecture \cite{AlonTarsi92} asserts that for even $n$, $L_n^{\Text{EVEN}}-L_n^{\Text{ODD}}\neq0$. Values of $L_n^{\Text{EVEN}}-L_n^{\Text{ODD}}$ for small $n$ can be found in \cite{StWan12}. Drisko \cite{Drisko97} proved the conjecture for $n=p+1$, where $p$ is an odd prime, and Glynn \cite{Glynn10} proved it for $n=p-1$. Since for odd $n$ $L_n^{\Text{EVEN}}=L_n^{\Text{ODD}}$ some extensions of this conjecture, that hold for odd $n$, were proposed, as will be described shortly.

A Latin square is called \emph{normalized} if its first row is the identity permutation, and \emph{unipotent} if all the elements of its main diagonal are equal. Let $U_n^{\Text{E}}$ and $U_n^{\Text{O}}$ be the numbers of normalized unipotent even and odd Latin squares, respectively. Zappa \cite{Zappa97} defined the Alon-Tarsi constant $AT(n):=U_n^{\Text{E}}-U_n^{\Text{O}}$ and introduced the following extension of the Alon-Tarsi conjecture:
\begin{conj}\label{conj1}
For all $n$, $AT(n)\neq0$
\end{conj}
 A Latin square is called \emph{reduced} if its first row and first column are the identity permutation. Let $R_n^\Text{E}$ and $R_n^\Text{O}$ denote the numbers of even and odd reduced Latin squares of order $n$, respectively. Another possible extension of the Alon-Tarsi conjecture is the following (see \cite{Kot2011} and \cite{StWan12}):
\begin{conj}\label{conj2}
For all $n$, $R_n^\Text{E}-R_n^\Text{O}\neq0$
\end{conj}
If $n$ is even these two conjectures are equivalent to the Alon-Tarsi conjecture. However, for odd $n$ it is not clear whether the two conjectures are equivalent, despite the existence of a bijection between reduced Latin squares and normalized unipotent Latin squares of order $n$ (see \cite{Zappa97}). Drisko \cite{Drisko98} proved Conjecture~\ref{conj1} in the case that $n$ is an odd prime. Conjecture~\ref{conj2} is known to be true for small values of $n$ (see \cite{StWan12}).

A Latin square $L$ of order $n$ determines $n$ permutation matrices $P_s$, $s\in[n]$, defined by $(P_s)_{ij}=1$ if and only if $L_{ij}=s$. Let $S_n$ be the collection of all $n\times n$ permutation matrices. For $P\in S_n$ let $\alpha_P$ be the corresponding permutation in $\Sym(n)$. The \emph{symbol-sign} $\epsilon_{\Text{sym}}(L)$ is the product of the $\epsilon(\alpha_{P_s})$, $s=1,\ldots,n$. A Latin square $L$ is \emph{symbol-even} if $\epsilon_{\Text{sym}}(L)=1$ and  \emph{symbol-odd} if $\epsilon_{\Text{sym}}(L)=-1$.

Let $X=(X_{ij})$ be the $n\times n$ matrix of indeterminates. The following theorem is due to MacMahon \cite{Mac98}:
\begin{thm}\label{thm0}
$L_n$ is the coefficient of $\prod_{i=1}^n\prod_{j=1}^nX_{ij}$ in $\per(X)^n$,
\end{thm}
where $\per(A)$ denotes the permanent of $A$. Stones \cite{Stones12} showed that if we replace permanent by determinant in the expression in Theorem~\ref{thm0} an expression for the Alon-Tarsi conjecture is obtained, namely
\begin{thm}\label{thm1}
$L_n^{\Text{EVEN}}-L_n^{\Text{ODD}}$ is the coefficient of $(-1)^{n(n-1)/2}\prod_{i=1}^n\prod_{j=1}^nX_{ij}$ in $\det(X)^n$.
\end{thm}
The idea of taking the $n^\Text{th}$ power of the determinant was used by Stones \cite{Stones12} to obtain another expression for $L_n^{\Text{EVEN}}-L_n^{\Text{ODD}}$:
\begin{thm}\label{thm1:5}
Let $B_n$ be the set of $n \times n$ $(0,1)$-matrices. For $A\in B_n$ let $\sigma_0(A)$ be the number of zero elements in $A$. Then
\begin{equation}\label{eq8}
    L_n^{\Text{EVEN}}-L_n^{\Text{ODD}}=(-1)^{\frac{n(n-1)}{2}}\sum_{A\in B_n}(-1)^{\sigma_0(A)}\det(A)^{n}
\end{equation}
\end{thm}

It will be shown in the Section~\ref{section2} that when $n$ is odd, ``hybrid'' expressions involving one permanent and $n-1$ determinants yield analogous results for $AT(n)$. In Section~\ref{section3} an alternative proof of Drisko's result that $AT(p)\neq 0$ for odd primes is shown. In Section~\ref{section4} a formula linking Conjectures~\ref{conj1} and \ref{conj2} is presented. Section~\ref{section5} introduces a formula relating the permanents of all distinct regular $p\times p$ adjacency matrices of bipartite graphs (up to renaming the vertices of one of the sides).
\section{Formulae for $AT(n)$}\label{section2}
For $\alpha\in\Sym(n)$ let $L_n^{\Text{SE}}(\alpha)$ (resp. $L_n^{\Text{SO}}(\alpha)$) be the number of symbol-even (resp. symbol-odd) Latin squares with $\alpha=\alpha_{P_1}$. Let $L_n^{\Text{CE}}(\alpha)$ (resp. $L_n^{\Text{CO}}(\alpha)$) be the number of column-even (resp. column-odd) Latin squares with $\alpha$ as the first column. Let $L_n^{\Text{CE}}(\alpha,\beta)$ (resp. $L_n^{\Text{CO}}(\alpha,\beta)$) be the number of column-even (resp. column-odd) Latin squares with $\alpha$ as the first row and $\beta$ as the first column. We have:
\begin{lem}\label{lem2}
If $n$ is odd then
\begin{equation*}
   \sum_{\pi\in\Sym(n)}\epsilon(\pi)(L_n^{\Text{SE}}(\pi)-L_n^{\Text{SO}}(\pi))=(-1)^{\frac{n(n-1)}{2}}n!(n-1)!AT(n)
\end{equation*}
\end{lem}
\begin{proof}
Viewing a Latin squares as a set of $n^2$ triples $(i,j,k)$, such that $L_{ij}=k$, and applying the mapping $\tau:(i,j,k)\rightarrow(i,k,j)$, the $k^{\Text{th}}$ column of $\tau(L)$ is the permutation $\alpha_{P_k}$ corresponding to the permutation matrix $P_k$ in $L$. Thus $L_n^{\Text{SE}}(\alpha)=L_n^{\Text{CE}}(\alpha)$ and $L_n^{\Text{SO}}(\alpha)=L_n^{\Text{CO}}(\alpha)$. We have:
\begin{equation*}
  \sum_{\pi\in\Sym(n)}\epsilon(\pi)(L_n^{\Text{SE}}(\pi)-L_n^{\Text{SO}}(\pi))
  = \sum_{\pi\in\Sym(n)}\epsilon(\pi)(L_n^{\Text{CE}}(\pi)-L_n^{\Text{CO}}(\pi))
\end{equation*}
By applying $\pi^{-1}$ to the columns of each Latin squares with $\pi$ as its first column we see that if $n$ is odd then
$\epsilon(\pi)(L_n^{\Text{CE}}(\pi)-L_n^{\Text{CO}}(\pi))=L_n^{\Text{CE}}(\id)-L_n^{\Text{CO}}(\id)$. Thus
\begin{equation*}
  \sum_{\pi\in\Sym(n)}\epsilon(\pi)(L_n^{\Text{SE}}(\pi)-L_n^{\Text{SO}}(\pi))
   = n!(L_n^{\Text{CE}}(\id)-L_n^{\Text{CO}}(\id)).
\end{equation*}
Since exchanging columns of a Latin square does not alter the column parity we have that for each $\beta\in\Sym(n)$ such that $\beta(1)=1$,  $L_n^{\Text{CE}}(\beta,\id)-L_n^{\Text{CO}}(\beta,\id)=L_n^{\Text{CE}}(\id,\id)-L_n^{\Text{CO}}(\id,\id)$. Thus
\begin{equation*}
\begin{split}
  \sum_{\pi\in\Sym(n)}\epsilon(\pi)(L_n^{\Text{SE}}(\pi)-L_n^{\Text{SO}}(\pi))
  & = n!(L_n^{\Text{CE}}(\id)-L_n^{\Text{CO}}(\id))\\
  & = n!\sum_{\substack{\beta\in\Sym(n)\\ \beta(1)=1}}L_n^{\Text{CE}}(\beta,\id)-L_n^{\Text{CO}}(\beta,\id)\\
  & = n!(n-1)!(L_n^{\Text{CE}}(\id,\id)-L_n^{\Text{CO}}(\id,\id))
\end{split}
\end{equation*}

We use the notation $R^{(+,-)}_n$ for the number of reduced Latin squares with even row parity and odd column parity ($R^{(+,+)}_n$, $R^{(-,+)}_n$ and $R^{(-,-)}_n$ are defined accordingly). Since $L_n^{\Text{CE}}(\id,\id)$ is the number of column-even reduced Latin squares, we have:
\begin{equation*}
\begin{split}
  L_n^{\Text{CE}}(\id,\id)-L_n^{\Text{CO}}(\id,\id) & = R^{(+,+)}_n+R^{(-,+)}_n-R^{(+,-)}_n-R^{(-,-)}_n \\
    & =R^{(+,+)}_n-R^{(-,-)}_n.
\end{split}
\end{equation*}
Since
\begin{equation*}
    AT(n)=
    \begin{cases}
       R^{(+,+)}_n-R^{(-,-)}_n,   & \Text{if $n\equiv0,1\pmod{4}$} \\
       R^{(-,-)}_n-R^{(+,+)}_n,   & \Text{if $n\equiv2,3\pmod{4}$},
    \end{cases}
\end{equation*}
by Section 5 in \cite{Zappa97}, the result follows.
\end{proof}
We now have a result, analogous to Theorem~\ref{thm1}, for $AT(n)$:
\begin{thm}\label{thm2}
Let $n$ be odd and let $X=(X_{ij})$ be the $n\times n$ matrix of indeterminates. Then $AT(n)$ is the coefficient of $(-1)^{\frac {n(n-1)}{2}}\prod_{i=1}^n\prod_{j=1}^nX_{ij}$ in $\frac{1}{n!(n-1)!}\per(X)\det(X)^{n-1}$.
\end{thm}
\begin{proof}
For $\Pee\in(S_n)^n$ let $\Pee=(P_1,P_2,\ldots,P_n)$ and for $s=1,\ldots,n$ let $\alpha_s=\alpha_{P_s}$. Expanding $\per(X)$ and $\det(X)$ we obtain
\begin{equation}\label{eq1}
    \per(X)\det(X)^{n-1}=\sum_{\pi\in \Sym(n)}
    \prod X_{i\pi(i)}\sum_{\substack{\Pee\in(S_n)^n\\ \pi=\alpha_1}}\prod_{s=2}^n\epsilon(\alpha_s)\prod_{k=1}^n X_{k\alpha_s(k)}.
\end{equation}
Now, for each $\pi\in\Sym(n)$ the number of square-free terms in
\begin{equation*}
    \prod X_{i\pi(i)}\sum_{\substack{\Pee\in(S_n)^n\\ \pi=\alpha_1}}\prod_{j=2}^n\epsilon(\alpha_j)\prod_{i=1}^n X_{i\alpha_j(i)}
\end{equation*}
is equal to $\epsilon(\pi)(L_n^{\Text{SE}}(\pi)-L_n^{\Text{SO}}(\pi))$. Hence, by (\ref{eq1}), the coefficient of $\prod_{i=1}^n\prod_{j=1}^nX_{ij}$ in $\per(X)\det(X)^{n-1}$ is
\begin{equation*}
    \sum_{\pi\in \Sym(n)}\epsilon(\pi)(L_n^{\Text{SE}}(\pi)-L_n^{\Text{SO}}(\pi)),
\end{equation*}
and the result follows from Lemma~\ref{lem2}.
\end{proof}
We also have an analogue of Theorem~\ref{thm1:5} for $AT(n)$:
\begin{thm}\label{thm3}
Let $B_n$ be the set of $n \times n$ $(0,1)$-matrices. For $A\in B_n$ let $\sigma_0(A)$ be the number of zero elements in $A$. If $n$ is odd then
\begin{equation}\label{eq8:1}
    AT(n)=\frac{(-1)^{\frac{n(n-1)}{2}}}{n!(n-1)!}\sum_{A\in B_n}(-1)^{\sigma_0(A)}\per(A)\det(A)^{n-1}
\end{equation}
\end{thm}
\begin{proof}
Most of the proof is similar to Stones' proof of Theorem~\ref{thm1:5}. By (\ref{eq1}),
\begin{equation}\label{eq9}
    \sum_{A\in B_n}(-1)^{\sigma_0(A)}\per (A)\det(A)^{n-1}=\sum_{(A,\Pee)\in B_n\times (S_n)^n}Z(A,\Pee)
\end{equation}
where
\begin{equation*}
    Z(A,\Pee)=(-1)^{\sigma_0(A)}\prod_{i=1}^n A_{i\alpha_1(i)}\prod_{s=2}^n\epsilon(\alpha_s)\prod_{k=1}^n A_{k\alpha_s(k)}.
\end{equation*}
If for $(A,\Pee)$ there exists $i,j\in [n]$ such that $(P_s)_{ij}=0$ for all $s=1,\ldots,n$, then let $A^c$ be the matrix formed by toggling $A_{ij}$ in the lexicographically first such coordinate $ij$. Thus $Z(A,\Pee)=-Z(A^c,\Pee)$ and these two terms cancel in the sum in (\ref{eq9}). So, on the right hand side of (\ref{eq9}) we are left only with $\sum_{\Pee\in S^*}\prod_{s=2}^n \epsilon(P_s)$, where $S^*=\{(P_1,\ldots,P_n):\sum_{s=1}^n sP_s \Text{ is a Latin square}\}$ and $A$ is the all-1 matrix. Now,
\begin{equation*}
\begin{split}
  \sum_{\Pee\in S^*}\prod_{s=2}^n \epsilon(\alpha_s) & = \sum_{\pi\in\Sym(n)}\epsilon(\pi)\sum_{\substack{\Pee\in S^* \\\alpha_{P_1}=\pi}}\prod_{s=1}^n \epsilon(\alpha_s)\\ & = \sum_{\pi\in\Sym(n)}\epsilon(\pi)\sum_{\substack{\Pee\in S^* \\\alpha_{P_1}=\pi}}\epsilon_\Text{sym}\left(\sum_{s=1}^n sP_s\right) \\
    & = \sum_{\pi\in\Sym(n)}\epsilon(\pi)(L_n^{\Text{SE}}(\pi)-L_n^{\Text{SO}}(\pi)),
\end{split}
\end{equation*}
and the result follows from Lemma~\ref{lem2}.
\end{proof}
\section{An alternative proof of Drisko's theorem}\label{section3}
The main result of this section (Corollary~\ref{cor1}) was first proved by Drisko \cite{Drisko98}. An alternative proof, based on the results of Section~\ref{section2}, is presented here. I am indebted to an anonymous reviewer for suggesting this proof.

In this section the rows and columns of an $n \times n$ matrix will be indexed by the numbers $0,1,\ldots,n-1$.
\begin{defn}
Let $A$ be an $n\times n$ matrix and Let $B$ be a subset of cells of $A$. Let $k$ be an integer. The $k$-left shift of $B$ is the set of cells $\{b_{i,(j-k)\bmod n}:b_{i,j}\in B\}$. The $k$-down shift of $B$ is the set of cells $\{b_{(i+k)\bmod n,j}:b_{i,j}\in B\}$.
\end{defn}

\begin{defn}
An $n\times n$ matrix $A$ will be said to be $k$-left row shifted, for $0<k<n$, if for all $i=1,\ldots,n-1$, the $i^{\Text{th}}$ row of $A$ is equal to the $k$-left shift of the $(i-1)^{\Text{st}}$ row, and the $0^{\Text{th}}$ row is equal to the $k$-left shift of the $(n-1)^{\Text{st}}$ row.
\end{defn}
\begin{rem}\label{rem1}
If $p$ is an odd prime and $A$ is a $p\times p$ $k$-left row shifted matrix, then the set of cells of $A$ is the disjoint union of $p$ diagonals, where the elements of each diagonal are all equal. These diagonals will be referred to as the \emph{principal diagonals} of $A$.
\end{rem}
\begin{lem}\label{lem3}
Let $p$ be an odd prime. Let $A$ be a $p\times p$ $k$-left row shifted (0,1)-matrix. Let $\mathbf{b}$ be the first row of $A$ and let $|\mathbf{b}|$ be the number of 1's in $\mathbf{b}$. Then
\begin{enumerate}
  \item [\Text{(i)}] $\per(A)\equiv|\mathbf{b}| \pmod{p}$
  \item [\Text{(ii)}] $\det(A)\equiv \pm|\mathbf{b}|\pmod{p}$
\end{enumerate}
\end{lem}
\begin{proof}
Part (i) can be easily obtained from Ryser's permanent formula (\cite{Rys63}, see also \url{http://mathworld.wolfram.com/RyserFormula.html}). However, a different approach, that will also apply to Part (ii), is used here. We define a mapping $s$ on the set of diagonals of $A$ as follows: For a diagonal $d$ in $A$, $s(d)$ is obtained by taking the $k$-left shift of $d$ and then taking the 1-down shift of the result. Note that the fixed points of $s$ are exactly the principal diagonals defined in Remark~\ref{rem1}.
The mapping $s$ is a bijection and, since $A$ is $k$-left row shifted, $s(d)$ contain the same set of values as $d$. In particular, if $d$ consists only of 1's, so does $s(d)$.
Also note that $s^p(d)=d$ for all $d$ and thus, since $p$ is prime, each orbit under $s$ is of size one or $p$.
As mentioned above, the orbits of size one are those containing the principal diagonal. Thus, $\per(A)\bmod{p}$ is equal to the number of principal diagonals consisting only of 1's, and since there are $|\mathbf{b}|$ such diagonal Part (i) follows.

For Part (ii), it remains to show that all principal diagonals correspond to permutations of the same parity and that $s$ preserves the parity of the permutation corresponding to the diagonal acted upon.
Let $d_1$ and $d_2$ be two diagonals, such that $d_1$ is the $k$-left shift of $d_2$.
This means that if $\pi_1$ and $\pi_2$ are the corresponding permutations, then $\pi_2=\nu^k \circ\pi_1$ (application from right to left), where $\nu=(12\ldots p)$, which is an even permutation, since $p$ is odd. If $d_1$ and $d_2$ are principal diagonals then $d_1$ is the $k$-left shift of $d_2$ for some $k$. Thus, all fixed diagonals correspond to permutations of the same parity.
If $d_1$ is the $k$-down shift of $d_2$, then the corresponding permutations satisfy $\pi_1=\pi_2\circ\nu^k$. Since $s$ consists of a left shift and a down shift, $s$ preserves the parity. This proves (ii).
\end{proof}
\begin{thm}\label{thm3:1}
Let $p$ be an odd prime. Let $B_p$ be the set of $p \times p$ (0,1)-matrices. Then
\begin{equation}\label{eq3:1}
    \frac{1}{p}\sum_{A\in B_p}(-1)^{\sigma_0(A)}\per(A)\det(A)^{p-1}\equiv-1\pmod{p}.
\end{equation}
\end{thm}
\begin{proof}
Define the group $G=\langle\nu\rangle\times\langle\nu\rangle$, where $\nu=(12\cdots p)$. The group $G$ acts on $B_p$ by permuting the rows and columns, so that for each element of $G$, its first component permutes the order of the rows and the second component permutes the order of the columns. By The Orbit-Stabilizer Theorem, an orbit has size $|G|=p^2$ unless each of its elements has a non-trivial stabilizer in $G$. If $g=(\nu^i,\nu^j)$ is a stabilizer of $A\in B_p$, so is any of its powers, including $(\nu,\nu^k)$ for some $k$, since $p$ is prime. Thus, an orbit has size smaller than $p^2$ if and only if for each matrix $A$ in that orbit there exists some $0<k<p$ for which $(\nu,\nu^k)A=A$. Let
\begin{equation*}
    D=\{A\in B_p| (\nu,\nu^k)A=A\Text{ for some }0<k<p\}.
\end{equation*}
The action of $G$ preserves $\sigma_0$ and, since $\nu$ is an even permutation, it also preserves the permanent and the determinant. We have
\begin{equation*}
    \frac{1}{p}\sum_{A\in B_p}(-1)^{\sigma_0(A)}\per(A)\det(A)^{p-1}\equiv
    \frac{1}{p}\sum_{A\in D}(-1)^{\sigma_0(A)}\per(A)\det(A)^{p-1}\pmod{p}.
\end{equation*}
Hence, it suffices to prove (\ref{eq3:1}) with ``$B_p$'' replaced by ``$D$''.

Suppose $(\nu,\nu^k)A=A$. Then, after applying $\nu^k$ to the $i^{\Text{th}}$ row the $(i+1)^{\Text{st}}$ is obtained, for $i=0,\dots,p-2$ and applying $\nu^k$ to the $(p-1)^{\Text{st}}$ row yields the $0^{\Text{th}}$ row.
This implies that $A$ is a $(p-k)$-left row shifted matrix.
Thus, $A$ is uniquely determined by its first row $\mathbf{b}$ and the number $k$.
We denote this by $A=A(\mathbf{b},k)$.

Now, suppose $A=A(\mathbf{b},k)$ is not the all-1 matrix and let $a=|\mathbf{b}|$. Since $p$ is odd, $\sigma_0(A)\equiv a\pmod{2}$. Then, by Lemma~\ref{lem3} and Fermat's Little Theorem,
$(-1)^{\sigma_0(A)}\per(A)\det(A)^{p-1}\equiv(-1)^aa\pmod{p}$.
For a fixed $a\in\{1,\ldots,p-1\}$, the number of distinct matrices $A(\mathbf{b},k)$ with $|\mathbf{b}|=a$ is $\binom{p}{a}(p-1)$. Therefore,
\begin{equation*}
    \frac{1}{p}\sum_{A\in D}(-1)^{\sigma_0(A)}\per(A)\det(A)^{p-1}\equiv
    \frac{1}{p}\sum_{a=1}^{p-1}\binom{p}{a}(p-1)(-1)^aa\pmod{p},
\end{equation*}
where the cases that $a\in\{0,p\}$ have been discarded since they correspond to the all-0 and all-1 matrices, which have zero determinant. The result now follows from the binomial identity
\begin{equation*}
    \sum_{a=0}^{p}\binom{p}{a}(-1)^aa=0
\end{equation*}
(see \url{http://en.wikipedia.org/wiki/Binomial_coefficient}).
\end{proof}
The following result was first proved by Drisko \cite{Drisko98}:
\begin{cor}\label{cor1}
If $p$ is an odd prime, then
\begin{equation*}
    AT(p)\equiv (-1)^{\frac{p-1}{2}}\pmod{p}.
\end{equation*}
\end{cor}
\begin{proof}
When $n=p$ is an odd prime we can rearrange (\ref{eq8}) to obtain
\begin{equation*}
\begin{split}
  AT(p) & = (-1)^{\frac{p-1}{2}}\times\frac{(-1)}{(n-1)!^2}\times\frac{1}{p}
  \sum_{A\in B_p}(-1)^{\sigma_0(A)}\per(A)\det(A)^{p-1}\\
    & \equiv (-1)^{\frac{p-1}{2}}\times (-1)\times(-1)\pmod{p},
\end{split}
\end{equation*}
by Theorem~\ref{thm3:1}. The result follows.
\end{proof}
\section{Linking conjectures~\ref{conj1} and \ref{conj2}}\label{section4}
The following statement is obtained as part of a proof in \cite{Kot2011}:
\begin{prop}\label{prop2}
Let $n$ be odd and let $A_1,A_2,\ldots,A_n$ be $n\times n$ matrices over a field. Then
\begin{equation}\label{eq30}
    \sum_{\substack{\rho,\sigma\in\Sym(n)^n\\\rho_1=\id}}
\epsilon(\sigma_1)\epsilon(\sigma)\epsilon(\rho)\prod_{i,j=1}^n (A_j)_{\sigma_i(j),\rho_j(i)}=(n-1)!\cdot (R_n^\Text{E}-R_n^\Text{O})\per(A_1)\prod_{j=2}^n \det(A_j).
\end{equation}
\end{prop}
Here $\rho_1$ and $\sigma_1$ are the first components in $\rho$ and $\sigma$ respectively. Combining Proposition~\ref{prop2} with Theorem~\ref{thm2} yields the following identity, linking $AT(n)$ and $R_n^\Text{E}-R_n^\Text{O}$:
\begin{thm}\label{thm4:1}
Let $X=(X_{ij})$ be an $n \times n$ matrix of indeterminates. Then $AT(n)\cdot(R_n^\Text{E}-R_n^\Text{O})$ is the coefficient of $(-1)^\frac{n(n-1)}{2}\prod_{i=1}^n\prod_{j=1}^nX_{ij}$ in
\begin{equation*}
    \frac{1}{n!(n-1)!^2}\sum_{\substack{\rho,\sigma\in\Sym(n)^n\\\rho_1=\id}}
\epsilon(\sigma_1)\epsilon(\sigma)\epsilon(\rho)\prod_{i,j=1}^n X_{\sigma_i(j),\rho_j(i)}.
\end{equation*}
\end{thm}
\begin{proof}
This follows by taking $A_1=A_2=\cdots=A_n=X$ in (\ref{eq30}) and applying Theorem~\ref{thm2}.
\end{proof}
Thus, showing that the above coefficient is nonzero would prove both conjectures.
\section{On the permanent of adjacency matrices}\label{section5}
The evaluation of the permanent of a matrix is a complex problem, even for adjacency matrices of bipartite graphs ((0,1)-matrices) (see \cite{Jerrum04}). Theorem~\ref{thm3} leads to an interesting identity involving the permanents of (0,1)-matrices:
\begin{thm}\label{thm4}
Let $p$ be an odd prime, let $B_p$ be the set of $p\times p$ $(0,1)$-matrices, and let $B_p^*=\{A\in B_p:\det(A)\not\equiv0\pmod{p}\}$. Let $B_p^\dag$ be a set of representatives in $B_p$ of the row permutation classes. Then
\begin{equation*}
    \sum_{A\in B_p^\dag\cap B_p^*}(-1)^{\sigma_0(A)}\per(A)\equiv -1 \pmod{p}.
\end{equation*}
\end{thm}
\begin{proof}
Let $B_p^r$ be the subset of $B_p$ containing the regular matrices. From (\ref{eq8}) we have:
\begin{equation*}
    AT(p)=\frac{(-1)^{\frac{p-1}{2}}}{p!(p-1)!}\sum_{A\in B_p^r}(-1)^{\sigma_0(A)}\per(A)\det(A)^{p-1}
\end{equation*}
If $A'$ can be obtained from $A$ by permuting the rows, then $\per(A')=\per(A)$ and $\det(A')^{p-1} =\det(A)^{p-1}$ (since $p-1$ is even). Since the rows of each $A\in B_p^r$ are all distinct, each row permutation class in $B_p^r$ contains exactly $p!$ matrices. Let $B_p^\dag$ be a set of representatives of the row permutation classes in $B_p$. Then
\begin{equation*}
    AT(p)=\frac{(-1)^{\frac{p-1}{2}}}{(p-1)!}\sum_{A\in B_p^\dag\cap B_p^r}(-1)^{\sigma_0(A)}\per(A)\det(A)^{p-1}.
\end{equation*}
By Fermat's little theorem and Wilson's theorem we have
\begin{equation*}
    AT(p)\equiv(-1)(-1)^{\frac{p-1}{2}}\sum_{A\in B_p^\dag\cap B_p^*}(-1)^{\sigma_0(A)}\per(A) \pmod{p}.
\end{equation*}
The result follows from Corollary~\ref{cor1}.
\end{proof}
\begin{rem}
If we view an $n\times n$ (0,1)-matrix $A$ as the adjacency matrix of a bipartite graph $G_A$, having two parts of identical size $n$, then $\per(A)$ is the number of perfect matchings in $G_A$. A set $B_p^\dag$, as in Theorem~\ref{thm4}, represents all possible such graphs, up to renaming the vertices of one of the parts.
\end{rem}
\bibliographystyle{amsplain}
\bibliography{AT_refs}
\end{document}